\documentclass[a4paper,twoside,12pt]{article}
\usepackage[english]{babel}

\usepackage{amsmath,amsfonts,amssymb,amsthm}
\newtheorem{teo}{Theorem}
\newtheorem{lem}{Lemma}

\newtheorem{remark}{Remark}

 \title{{\bf  Schatten index of the operator  via the   real component of its inverse    }}

\author{Maksim \,V.~Kukushkin   \\ \\
 \small  \textit{HSE University, 101000,  Moscow, Russia}\\
 \textit{\small\textit{kukushkinmv@rambler.ru}} }

\date{}
\begin{document}

\maketitle

\begin{abstract} In this paper we study spectral properties  of   non-selfadjoint operators with the discrete spectrum.  The main challenge is to represent a complete description  of belonging to the Schatten  class  through  the properties  of the Hermitian real component.  The method of estimating the singular values is elaborated by virtue of the established asymptotic formulas. The latter fundamental result is advantageous since many theoretical statements based upon it,  one of them is a concept on the root vectors series expansion which leads to    a wide spectrum of applications in the theory of evolution equations. In this regard the evolution equations of fractional order with the operator in the term not containing the time variable are involved.   The concrete well-known operators are considered and the advantage of the represented method is convexly shown.

\end{abstract}
\begin{small}\textbf{Keywords:}
 Strictly accretive operator;  Abel-Lidskii basis property;   Schatten-von Neumann  class; convergence exponent; counting function.   
\end{small}

\section{Introduction}
The idea to write this paper origins from the concept of decomposition of an element of the abstract Hilbert space on the root vectors series. The latter concept lies in the framework of abstract  functional  analysis and its appearance arises from  elaboration of  methods of solving  evolution equations investigated in the recent century by Lidskii V.B., Matsaev V.I., Agaranovich M.S. and others.   In its simple reduced form, applicably to selfadjoint operators   the concept admits the   interpretation through   the  well-known fact that the   eigenvectors   of the compact selfadjoint operator form a basis in the closure of its range. The question what happens in the case when the operator is non-selfadjoint is rather complicated and  deserves to be considered as a separate part of the spectral theory.

 We should make a brief digression and explain that  a relevance appears just   in the case when   a senior term of a considered  operator is not selfadjoint  for there is a number of papers \cite{firstab_lit:1Katsnelson},\cite{firstab_lit:1Krein},\cite{firstab_lit:Markus Matsaev},\cite{firstab_lit:2Markus},\cite{firstab_lit:Motovilov},
\cite{firstab_lit:Shkalikov A.} devoted to the perturbed selfadjoint operators.  The fact is that most of them deal with a decomposition of the  operator  on a sum,  where the senior term  must be either a selfadjoint or normal operator. In other cases, the  methods of the papers
     \cite{kukushkin2019}, \cite{firstab_lit(arXiv non-self)kukushkin2018} become relevant   and allow us  to study spectral properties of   operators  whether we have the mentioned above  representation or not, moreover   they have a natural mathematical origin that appears brightly  while we are considering abstract constructions expressed  in terms of the semigroup theory   \cite{kukushkin2021a}.

  Generally, the aim of the mentioned  part of the spectral theory   are   propositions on the convergence of the root vector series in one or another sense to an element belonging to the closure of the operator range, we   mean   Bari, Riesz,  Abel-Lidskii senses of the series convergence  \cite{firstab_lit:1Gohberg1965}.
  The main condition in terms of which the propositions are mostly described is the asyptotics of the operator  singular numbers, here we should note that  originally  it is  formulated in terms of   the operator beloning to the Schatten class.  However, Agaranovich M.S. made an attempt to express the sufficient conditions of the root vector series basis property, in the mentioned above generalized sense, through the asymptotics of the  eigenvalues of the real component \cite{firstab_lit:2Agranovich1994}. The paper by    Markus A.S.,  Matsaev V.I. \cite{firstab_lit:Markus Matsaev}  can be also considered within the scope  since it establishes the relationship between  the asymptotics  of the operator eigenvalues absolute value  and eigenvalues of the  real component.

  Thus, the interest how to express root vectors series decomposition theorems through the asymptotics of the real component eigenvalues  arose  previously    what along the obvious technical advantage in finding the asymptotics  create a prerequisite to investigate the matter properly.  We should point out that under the desired relationship between asymptotics we are able to reformulate theorems on the root vectors  series expansion  in terms of the assumptions related to the real component of the operator. The latter idea is relevant  since in many cases the calculation of the real component eigenvalues  asymptotics   simpler than direct calculation of the singular numbers asymptotics.

 If we make a comparison analysis between the methods of root vectors decomposition by Lidskii V.B. and Agaranovich M.S. we will see that the first one formulated the conditions in terms of the singular values but the second one did  in terms of the real component eigenvalues. In this regard we will show that the real component eigenvalue  asymptotics stronger than the one of the singular numbers, however  Agaranovich M.S. imposed the additional condition - spectrum belongs to  the domain of the parabolic type. From the latter point of view the results by Lidskii V.B. more advantageous since   the convergence in the Abel-Lidskii sense  was established for an operator class wider than the class of sectorial operators. Apparently,   a reasonable question that may   appear  is about minimal conditions that guaranty the desired result what in particular is considered in this paper.

  Here, we obviously can extend  the  results devoted to operators with the discrete spectra to operators with the compact resolvent, for they can be easily reformulated from one realm to another. In this regard, we should make  warning that the latter fact does not hold for real components since the real component of the inverse  operator does not coincide with the inverse of the operator real component. However, such a complication was diminished due to the results
   \cite{firstab_lit(arXiv non-self)kukushkin2018},\cite{kukushkin2019}, where the asymptotic equivalence between the eigenvalues of the mentioned operators was  easteblished.

 A couple of words on the applied relevance of the issue. The  abstract approach to the Cauchy problem for  the fractional evolution equation is classic one \cite{firstab_lit:Bazhl},\cite{Ph. Cl}. In its framework,  the application  of results connected with the basis property  covers  many   problems of the theory of evolution  equations \cite{firstab_lit:1Lidskii}, \cite{kukushkin2019}, \cite{kukushkin2021a}, \cite{firstab_lit:1kukushkin2021}, \cite{firstab_lit:2kukushkin2022}, \cite{firstab_lit(axi2022)}. In its general statement the  problem appeals to many applied ones,  we can   produce a number  of papers dealing  with differential equations which can be studied by the  abstract methods \cite{L. Mor}, \cite{firstab_lit:Mainardi F.}, \cite{firstab_lit:Mamchuev2017a}, \cite{firstab_lit:Mamchuev2017}, \cite{firstab_lit:Pskhu}, \cite{firstab_lit:Wyss}. Apparently, the  main advantage of this paper is a a technique that allows to implement the verification of the abstract conditions of the exitance and uniqueness theorem  for the concrete evolution equations. Thus, we can  claim that the offered  approach is undoubtedly novel from the abstract theory point of view  and relevant from the applied one.

\section{Preliminaries}

Let    $ C,C_{i} ,\;i\in \mathbb{N}_{0}$ be   real constants. We   assume   that  a  value of $C$ is positive and   can be different in   various formulas  but   values of $C_{i} $ are  certain. Denote by $ \mathrm{int} \,M,\;\mathrm{Fr}\,M$ the interior and the set of boundary points of the set $M$ respectively.   Everywhere further, if the contrary is not stated, we consider   linear    densely defined operators acting on a separable complex  Hilbert space $\mathfrak{H}$. Denote by $ \mathcal{B} (\mathfrak{H})$    the set of linear bounded operators   on    $\mathfrak{H}.$  Denote by
    $\tilde{L}$   the  closure of an  operator $L.$ We establish the following agreement on using  symbols $\tilde{L}^{i}:= (\tilde{L})^{i},$ where $i$ is an arbitrary symbol.  Denote by    $    \mathrm{D}   (L),\,   \mathrm{R}   (L),\,\mathrm{N}(L)$      the  {\it domain of definition}, the {\it range},  and the {\it kernel} or {\it null space}  of an  operator $L$ respectively. The deficiency (codimension) of $\mathrm{R}(L),$ dimension of $\mathrm{N}(L)$   are denoted by $\mathrm{def}\, L,\;\mathrm{nul}\,L$ respectively.  Assume that $L$ is a closed   operator acting on $\mathfrak{H},\,\mathrm{N}(L)=0,$  let us define a Hilbert space
$
 \mathfrak{H}_{L}:= \big \{f,g\in \mathrm{D}(L),\,(f,g)_{ \mathfrak{H}_{L}}=(Lf,Lg)_{\mathfrak{H} } \big\}.
$
Consider a pair of complex Hilbert spaces $\mathfrak{H},\mathfrak{H}_{+},$ the notation
$
\mathfrak{H}_{+}\subset\subset\mathfrak{ H}
$
   means that $\mathfrak{H}_{+}$ is dense in $\mathfrak{H}$ as a set of    elements and we have a bounded embedding provided by the inequality
$$
\|f\|_{\mathfrak{H}}\leq C_{0}\|f\|_{\mathfrak{H}_{+}},\,C_{0}>0,\;f\in \mathfrak{H}_{+},
$$
moreover   any  bounded  set with respect to the norm $\mathfrak{H}_{+}$ is compact with respect to the norm $\mathfrak{H}.$
  Let $L$ be a closed operator, for any closable operator $S$ such that
$\tilde{S} = L,$ its domain $\mathrm{D} (S)$ will be called a core of $L.$ Denote by $\mathrm{D}_{0}(L)$ a core of a closeable operator $L.$ Let    $\mathrm{P}(L)$ be  the resolvent set of an operator $L$ and
     $ R_{L}(\zeta),\,\zeta\in \mathrm{P}(L),\,[R_{L} :=R_{L}(0)]$ denotes      the resolvent of an  operator $L.$ Denote by $\lambda_{i}(L),\,i\in \mathbb{N} $ the eigenvalues of an operator $L.$
 Suppose $L$ is  a compact operator and  $N:=(L^{\ast}L)^{1/2},\,r(N):={\rm dim}\,  \mathrm{R}  (N);$ then   the eigenvalues of the operator $N$ are called   the {\it singular  numbers} ({\it s-numbers}) of the operator $L$ and are denoted by $s_{i}(L),\,i=1,\,2,...\,,r(N).$ If $r(N)<\infty,$ then we put by definition     $s_{i}=0,\,i=r(N)+1,2,...\,.$
 According  to the terminology of the monograph   \cite{firstab_lit:1Gohberg1965}  the  dimension  of the  root vectors subspace  corresponding  to a certain  eigenvalue $\lambda_{k}$  is called  the {\it algebraic multiplicity} of the eigenvalue $\lambda_{k}.$
Let  $\nu(L)$ denotes   the sum of all algebraic multiplicities of an  operator $L.$ Denote by $n(r)$ a function equals to a number of the elements of the sequence $\{a_{n}\}_{1}^{\infty},\,|a_{n}|\uparrow\infty$ within the circle $|z|<r.$ Let $A$ be a compact operator, denote by $n_{A}(r)$   {\it counting function}   a function $n(r)$ corresponding to the sequence  $\{s^{-1}_{i}(A)\}_{1}^{\infty}.$ Let  $\mathfrak{S}_{p}(\mathfrak{H}),\, 0< p<\infty $ be       a Schatten-von Neumann    class and      $\mathfrak{S}_{\infty}(\mathfrak{H})$ be the set of compact operators.
    Suppose  $L$ is  an   operator with a compact resolvent and
$s_{n}(R_{L})\leq   C \,n^{-\mu},\,n\in \mathbb{N},\,0\leq\mu< \infty;$ then
 we
 denote by  $\mu(L) $   order of the     operator $L$ in accordance with  the definition given in the paper  \cite{firstab_lit:Shkalikov A.}.
 Denote by  $ \mathfrak{Re} L  := \left(L+L^{*}\right)/2,\, \mathfrak{Im} L  := \left(L-L^{*}\right)/2 i$
  the  real  and   imaginary components    of an  operator $L$  respectively.
In accordance with  the terminology of the monograph  \cite{firstab_lit:kato1980} the set $\Theta(L):=\{z\in \mathbb{C}: z=(Lf,f)_{\mathfrak{H}},\,f\in  \mathrm{D} (L),\,\|f\|_{\mathfrak{H}}=1\}$ is called the  {\it numerical range}  of an   operator $L.$
  An  operator $L$ is called    {\it sectorial}    if its  numerical range   belongs to a  closed
sector     $\mathfrak{ L}_{\iota}(\theta):=\{\zeta:\,|\arg(\zeta-\iota)|\leq\theta<\pi/2\} ,$ where      $\iota$ is the vertex   and  $ \theta$ is the semi-angle of the sector   $\mathfrak{ L}_{\iota}(\theta).$ If we want to stress the  correspondence  between $\iota$ and $\theta,$  then   we will write $\theta_{\iota}.$
 An operator $L$ is called  {\it bounded from below}   if the following relation  holds  $\mathrm{Re}(Lf,f)_{\mathfrak{H}}\geq \gamma_{L}\|f\|^{2}_{\mathfrak{H}},\,f\in  \mathrm{D} (L),\,\gamma_{L}\in \mathbb{R},$  where $\gamma_{L}$ is called a lower bound of $L.$ An operator $L$ is called  {\it   accretive}   if  $\gamma_{L}=0.$
 An operator $L$ is called  {\it strictly  accretive}   if  $\gamma_{L}>0.$      An  operator $L$ is called    {\it m-accretive}     if the next relation  holds $(A+\zeta)^{-1}\in \mathcal{B}(\mathfrak{H}),\,\|(A+\zeta)^{-1}\| \leq   (\mathrm{Re}\zeta)^{-1},\,\mathrm{Re}\zeta>0. $
An operator $L$ is called    {\it m-sectorial}   if $L$ is   sectorial    and $L+ \beta$ is m-accretive   for some constant $\beta.$   An operator $L$ is called     {\it symmetric}     if one is densely defined and the following  equality  holds $(Lf,g)_{\mathfrak{H}}=(f,Lg)_{\mathfrak{H}},\,f,g\in   \mathrm{D}  (L).$

Everywhere further,   unless  otherwise  stated,  we   use  notations of the papers   \cite{firstab_lit:1Gohberg1965},  \cite{firstab_lit:kato1980},  \cite{firstab_lit:kipriyanov1960}, \cite{firstab_lit:1kipriyanov1960},
\cite{firstab_lit:samko1987}.\\

\noindent{\bf  Sectorial sesquilinear forms and the Hermitian  components}\\

Consider the Hermitian    components of the operator (not necessarily bounded)
$$
\mathfrak{Re} L:=\frac{L+L^{\ast}}{2},\;\;\mathfrak{Im} L:=\frac{L-L^{\ast}}{2i},
$$
it is clear that in the case when the operator $L$ is unbounded but densely defined we need agreement of the domain of definitions of the operator and its adjoint, since in other case the real component may be not densely defined. However, the latter claim required concrete examples.

  Consider a   sesquilinear form   $ t  [\cdot,\cdot]$ (see \cite{firstab_lit:kato1980} )
defined on a linear manifold  of the Hilbert space $\mathfrak{H}.$   Denote by $   t  [\cdot ]$ the  quadratic form corresponding to the sesquilinear form $ t  [\cdot,\cdot].$
Let
$$
\mathfrak{h}=( t + t ^{\ast})/2,\, \mathfrak{k}   =( t - t ^{\ast})/2i
$$
   be a   real  and    imaginary component     of the   form $  t $ respectively, where $ t^{\ast}[u,v]=t \overline{[v,u]},\;\mathrm{D}(t ^{\ast})=\mathrm{D}(t).$ In accordance with the definitions, we have
    $
 \mathfrak{h}[\cdot]=\mathrm{Re}\,t[\cdot],\,  \mathfrak{k}[\cdot]=\mathrm{Im}\,t[\cdot].$ Denote by $\tilde{t}$ the  closure   of a   form $t.$  The range of a quadratic form
  $ t [f],\,f\in \mathrm{D}(t),\,\|f\|_{\mathfrak{H}}=1$ is called    {\it range} of the sesquilinear form  $t $ and is denoted by $\Theta(t).$
 A  form $t$ is called    {\it sectorial}    if  its    range  belongs to   a sector  having  a vertex $\iota$  situated at the real axis and a semi-angle $0\leq\theta_{\iota}<\pi/2.$   Suppose   $t$ is a closed sectorial form; then  a linear  manifold  $\mathrm{D}_{0}(t) \subset\mathrm{D} (t)$   is
called    {\it core}  of $t,$ if the restriction   of $t$ to   $\mathrm{D}_{0}(t)$ has the   closure
$t$ (see\cite[p.166]{firstab_lit:kato1980}).

Suppose  $L$ is a sectorial densely defined operator and $t[u,v]:=(Lu,v)_{\mathfrak{H}},\,\mathrm{D}(t)=\mathrm{D}(L);$  then
 due to   Theorem 1.27 \cite[p.318]{firstab_lit:kato1980}   the corresponding  form $t$ is   closable, due to
   Theorem 2.7 \cite[p.323]{firstab_lit:kato1980} there exists   a unique m-sectorial operator   $T_{\tilde{t}}$   associated  with  the form $\tilde{t}.$  In accordance with the  definition \cite[p.325]{firstab_lit:kato1980} the    operator $T_{\tilde{t}}$ is called     a {\it Friedrichs extension} of the operator $L.$

Due to Theorem 2.7 \cite[p.323]{firstab_lit:kato1980}  there exist unique    m-sectorial operators  $T_{t},T_{ \mathfrak{h}} $  associated  with   the  closed sectorial   forms $t,  \mathfrak{h}$   respectively.   The operator  $T_{  \mathfrak{h}} $ is called  a {\it real part} of the operator $T_{t}$ and is denoted in accordance with the original definition \cite{firstab_lit:kato1980} by  $\mathrm{Re}\, T_{t}.$

Here, we should stress that the construction of the real part in some cases  is obviously coincided with the one of the real component, however the latter does not require the agreement between the domain of definitions mentioned above. The conditions represented bellow reflect the nature of the uniformly elliptic operators being the direct generalization of the latter considered in the context of the Sobolev spaces.\\

 \noindent  $ (\mathrm{H}1) $ There  exists a Hilbert space $\mathfrak{H}_{+}\subset\subset\mathfrak{ H}$ and a linear manifold $\mathfrak{M}$ that is  dense in  $\mathfrak{H}_{+}.$ The operator $L$ is defined on $\mathfrak{M}.$    \\

 \noindent  $( \mathrm{H2} )  \,\left|(Lf,g)_{\mathfrak{H}}\right|\! \leq \! C_{1}\|f\|_{\mathfrak{H}_{+}}\|g\|_{\mathfrak{H}_{+}},\,
      \, \mathrm{Re}(Lf,f)_{\mathfrak{H}}\!\geq\! C_{2}\|f\|^{2}_{\mathfrak{H}_{+}} ,\,f,g\in  \mathfrak{M},\; C_{1},C_{2}>0.
$
\\

Consider  a condition  $\mathfrak{M}\subset \mathrm{D}( W ^{\ast}),$ in this case the real Hermitian component  $\mathcal{H}:=\mathfrak{Re }\,W$ of the operator is defined on $\mathfrak{M},$ the fact is that $\tilde{\mathcal{H}}$ is selfadjoint,    bounded  from bellow (see Lemma  3 \cite{firstab_lit(arXiv non-self)kukushkin2018}).  Hence a corresponding sesquilinear  form (denote this form by $h$) is symmetric and  bounded from bellow also (see Theorem 2.6 \cite[p.323]{firstab_lit:kato1980}). The conditions H1,H2 allow us to claim that   the form $t$ corresponding to the operator $W$ is a closed sectorial form, consider the corresponding form $\mathfrak{h}$.  It can be easily shown  that $h\subset   \mathfrak{h},$  but using this fact    we cannot claim in general that $\tilde{\mathcal{H}}\subset H,$ where $H:=\mathrm{Re}W$ (see \cite[p.330]{firstab_lit:kato1980} ). We just have an inclusion   $\tilde{\mathcal{H}}^{1/2}\subset H^{1/2}$     (see \cite[p.332]{firstab_lit:kato1980}). Note that the fact $\tilde{\mathcal{H}}\subset H$ follows from a condition $ \mathrm{D}_{0}(\mathfrak{h})\subset \mathrm{D}(h) $ (see Corollary 2.4 \cite[p.323]{firstab_lit:kato1980}).
 However, it is proved (see proof of Theorem  4 \cite{firstab_lit(arXiv non-self)kukushkin2018}) that relation H2 guaranties that $\tilde{\mathcal{H}}=H.$ Note that the last relation is very useful in applications, since in most concrete cases we can find a concrete form of the operator $\mathcal{H}.$

\vspace{0.5cm}
\noindent{\bf Previously obtained results}\\

Here, we represent  previously obtained results that will be undergone to  a thorough study since our principal challenge is to obtain an accurate description of the  Schatten-von Neumann class index of a non-selfadjoint operator.

Further, we consider the Theorem 1 \cite{kukushkin2021a}  statements separately under assumptions $\mathrm{H}1,\mathrm{H}2.$ Let $W$ be a restriction of the operator  $L$ on the set $\mathfrak{M},$ without loss of generality of reasonings, we assume that $W$ is closed since the conditions  H1, H2  guaranty  that it is closeable, more detailed information in this regard is given in  \cite{kukushkin2021a}.

 We  have the following classification in terms of the operator order $\mu,$ where it is defined as follows $ \lambda_{n}(R_{H})=O  (n^{-\mu}),\,n\rightarrow\infty.$\\

\noindent $({\bf A})$  The following Schatten classification holds
 $$\;
 R_{ \tilde{W} }\in  \mathfrak{S}_{p},\,\inf p\leq2/\mu,\,\mu\leq1,\; R_{ W }\in  \mathfrak{S}_{1},\,\mu>1.
 $$
 Moreover, under assumptions
$\lambda_{n}(R_{H})\geq  C \,n^{-\mu} ,\,0\leq\mu<\infty,$   the following implication holds $R_{  W}\in\mathfrak{S}_{p},\,p\in [1,\infty),\Rightarrow   \mu>1/p.$    \\

 Observe  that the given above classification is far from the exact description of the Schatten-von Neumann class index. However, having analyzed  the above   implications,  we can say that   it makes a prerequisite to establish  a hypotheses $ R_{ W }\in  \mathfrak{S}_{p},\,\inf p=1/\mu. $
The following narrative is devoted to its verification.

 Let  us undergone the technical tools involved in the proof of the statement to   the thorough analysis in order to absorb and contemplate them.
Consider the statement, if    $\mu\leq1,$ then   $\; R_{  W  }\in  \mathfrak{S}_{p},\,\inf p\leq2/\mu.$ The main result on which it is based is in the asymptotic equivalence between the the inverse of the real component and the real component of the resolvent, the latter due to the technical tool makes the result, i.e.
$$
(|R_{ W  }|^{2} f,f)_{\mathfrak{H}}=\|R_{  W }f\|^{2}_{\mathfrak{H}}\leq C \cdot{\rm Re}(R_{ W }f,f)_{\mathfrak{H}}= C\cdot \left( \mathfrak{Re}R_{ W }f,f\right)_{\mathfrak{H}}.
$$

Consider the statement, if  $\lambda_{n}(R_{H})\geq  C \,n^{-\mu},\,0\leq\mu<\infty,$    then  the following implication holds $R_{  W}\in\mathfrak{S}_{p},\,p\in [1,\infty),\Rightarrow \mu>1/p.$ The main results that guaranty the fulfilment of the latter are  inequality  (7.9) \cite[p.123]{firstab_lit:1Gohberg1965}, Theorem 3.5 \cite{kukushkin2021a}, in accordance with which,   we get
$$
\sum\limits_{i=1}^{\infty}|s_{i} (R_{W} )|^{p}\geq\sum\limits_{i=1}^{\infty}| (R_{W}\varphi_{i},\varphi_{i})_{\mathfrak{H}}|^{p}\geq\sum\limits_{i=1}^{\infty}|{\rm Re}(R_{W}\varphi_{i},\varphi_{i})_{\mathfrak{H}}|^{p}=
$$
$$
=\sum\limits_{i=1}^{\infty}| \left(\mathfrak{Re}R_{ W } \varphi_{i},\varphi_{i}\right)_{\mathfrak{H}}|^{p}=\sum\limits_{i=1}^{\infty}
|\lambda_{i}\left(\mathfrak{Re}R_{ W }\right)|^{p}\geq C   \sum\limits_{i=1}^{\infty}  i^{- \mu p },\,p\geq1.
$$

Bellow, we represent the second statement of Theorem 1 \cite{kukushkin2021a}, where the peculiar result related to the asymptotics of the absolute value of the eigenvalue is given\\

\noindent$({\bf B})$     In the case  $\nu(R_{ W })=\infty,\,\mu \neq0,$  the following relation  holds
$$
|\lambda_{n}(R_{W})|=  o\left(n^{-\tau}    \right)\!,\,n\rightarrow \infty,\;0<\tau<\mu.
$$
It is based on the Theorem 6.1 \cite[p.81]{firstab_lit:1Gohberg1965},  in accordance with which, we have
 \begin{equation}\label{3.31}
\sum\limits_{m=1}^{k}|{\rm Im}\,\lambda_{m}(L)|^{p}\leq \sum\limits_{m=1}^{k}
|\lambda_{m} (\mathfrak{Im} L   )|^{p},\;\left(k=1,\,2,...,\,\nu_{\,\mathfrak{I}}(L)\right),\,1\leq p<\infty,
\end{equation}
where    $ \nu_{\,\mathfrak{I}}(L) \leq \infty $ is   the sum of   all  algebraic multiplicities corresponding to   the not real  eigenvalues   of the operator $L$ (see  \cite[p.79]{firstab_lit:1Gohberg1965}).

Note that the statement (B)  will allow us to arrange brackets in the series that converges in the Abell-Lidskii sense  what would be an advantageous achievement in the later constructed theory.  However, it may be interesting if we do not have the  exact index of the Schatten class  for in this case, we obtain the obvious
$$
R_{W}\in \mathfrak{S}_{p},\Rightarrow s_{n}=o(n^{-1/p}),
$$
hence in accordance with the connection of the asymptotics (see Chapter II, \S3 \cite{firstab_lit:1Gohberg1965} ), we get $|\lambda_{n}(R_{W})|=o\left(n^{-1/p}    \right)$ that is the same if we have $p>1/\mu.$ Thus, along the mentioned above implication $R_{  W}\in\mathfrak{S}_{p},\,p\in [1,\infty),\Rightarrow p>1/\mu$ it makes the prerequisite to observe  the hypotheses $\inf p = 1/\mu. $

Apparently, the used technicalities appeal to the so-called non-direct estimates for singular numbers realized due to the series estimates. As we will see further, the main advantage of the series estimation is the absence of the conditions imposed on the type of the asymptotics, it may be not one of the power type. However, we will show that under the restriction imposed on the type of the asymptotics, assuming that one is of the power type, we can obtain direct estimates for singular numbers. In the reminder, let us note that classes of differential operators have the asymptotics of the power type what make the issue quite relevant.

\section{Main results}
\vspace{0.5cm}
\noindent{\bf 1. The main refinement of the result  A }\\

The reasonings produced bellow appeals to a compact operator $B$ what represents a most general case in the framework of the decomposition on the root vectors theory, however to obtain more peculiar results we are compelled to deploy some restricting conditions. In this regard we involve hypotheses H1,H2 if it is necessary.
The result represented bellow gives us  the upper estimate for the singular numbers it is based on the result  by  Ky Fan \cite{firstab_lit:Fan} which  can be found as a corollary of the well-known Allakhverdiyev theorem,  see   Corollary  2.2 \cite{firstab_lit:1Gohberg1965}.

\begin{lem}\label{L1} Assume that  $B$ is a compact sectorial operator with the vertex situated at the point zero, then
$$
s_{2m-1}(B)\leq \sqrt{2}\,\sec  \theta \cdot \lambda_{m} (\mathfrak{Re}  B),\;\;s_{2m}(B)\leq \sqrt{2}\,\sec  \theta \cdot  \lambda_{m} (\mathfrak{Re}  B), \;\;m =1,2,...\,  .
$$
\end{lem}
\begin{proof}
 Consider the Hermitian components
$$
\mathfrak{Re} B:=\frac{B+B^{\ast}}{2},\;\;\mathfrak{Im} B:=\frac{B-B^{\ast}}{2i},
$$
it is clear that they are compact selfadjoint operators, since $B$ is compact and due to the technicalities of the given algebraic constructions.
Note that the  following relation can be established by direct calculation
$$
\mathfrak{Re}^{2} \!B+ \mathfrak{Im}^{2} \!B=\frac{B^{\ast}B+BB^{\ast}}{2},
$$
from what follows the inequality
\begin{equation}\label{eq4a}
 \frac{1}{2} \cdot B^{\ast}B \leq     \mathfrak{Re}^{2} \!B+  \mathfrak{Im}^{2}B .
\end{equation}
Having analyzed the latter formula, we see that  it is rather reasonable to think over the opportunity of applying the corollary of the  minimax principle pursuing the aim to estimate the singular numbers of the operator $B.$ For the purpose of implementing the latter concept,   consider the following relation
$
\mathfrak{Re}^{2} \!B \,f_{n}=\lambda^{2}_{n} f_{n},
$
where $f_{n},\lambda_{n}$ are the eigenvectors and the eigenvalues of the operator $\mathfrak{Re}   B$ respectively.  Since the operator $\mathfrak{Re}   B$ is selfadjoint and compact then its set of eigenvalues form a basis in $\overline{\mathrm{R}(\mathfrak{Re}   B)}.$   Assume that there exists a non-zero  eigenvalue of the operator $ \mathfrak{Re}^{2} \!B $ that is different from  $\{\lambda^{2}_{n}\}_{1}^{\infty},$  then, in accordance with the well-known fact of the operator theory,  the corresponding eigenvector is orthogonal to the eigenvectors of the operator $\mathfrak{Re}   B.$  Taking into account the fact that the latter form a basis in $\overline{\mathrm{R}(\mathfrak{Re}   B)},$ we come to the conclusion that the eigenvector does not belong to  $\overline{\mathrm{R}(\mathfrak{Re}   B)}.$   Thus, the obtained contradiction proves the fact
$
\lambda_{n}(\mathfrak{Re}^{2} \!B)=\lambda^{2}_{n}(\mathfrak{Re}   B).
$
Implementing the same reasonings, we obtain
$
\lambda_{n}(\mathfrak{Im}^{2} \!B)=\lambda^{2}_{n}(\mathfrak{Im}   B).
$

Further, we   need a result by Ky Fan \cite{firstab_lit:Fan}  see   Corollary  2.2 \cite{firstab_lit:1Gohberg1965} (Chapter II, $\S$ 2.3), in accordance with which, we have
$$
s_{m+n-1}(\mathfrak{Re}^{2} \!B+ \mathfrak{Im}^{2}B)\leq \lambda_{m}(\mathfrak{Re}^{2} \!B)+\lambda_{n}(\mathfrak{Im}^{2}B),\;\;m,n=1,2,...\,.
$$
Choosing $n=m$ and $n=m+1,$  we obtain respectively
$$
s_{2m-1}(\mathfrak{Re}^{2} \!B+ \mathfrak{Im}^{2}B)\leq \lambda_{m}(\mathfrak{Re}^{2} \!B)+\lambda_{m}(\mathfrak{Im}^{2}B),
$$
$$
\, s_{2m}(\mathfrak{Re}^{2} \!B+ \mathfrak{Im}^{2}B)\leq \lambda_{m}(\mathfrak{Re}^{2} \!B)+\lambda_{m+1}(\mathfrak{Im}^{2}B) \;\;m =1,2,...\,.
$$
At this stage of the reasonings we need involve  the sectorial property $\Theta(B)\subset \mathfrak{L}_{0}(\theta)$ which  gives us $|\mathrm{Im}( Bf,f)|\leq \tan \theta \,\mathrm{Re}(   Bf,f).$  Applying the corollary of the minimax principle to the latter relation, we   get $|\lambda_{n}(\mathfrak{Im}   B)|\leq \tan \theta \, \lambda_{n}(\mathfrak{Re}   B).$ Therefore
$$
s_{2m-1}(\mathfrak{Re}^{2} \!B+ \mathfrak{Im}^{2}B)\leq \lambda_{m}(\mathfrak{Re}^{2} \!B)+\lambda_{m}(\mathfrak{Im}^{2}B)\leq\sec^{2}\!\theta \cdot \lambda_{m}^{2}(\mathfrak{Re}  B),
$$
$$
\, s_{2m}(\mathfrak{Re}^{2} \!B+ \mathfrak{Im}^{2}B)\leq \sec^{2}\!\theta  \cdot\lambda_{m}^{2}(\mathfrak{Re}  B) \;\;m =1,2,...\,.
$$
Applying the minimax principle to the formula \eqref{eq4a}, we get
$$
s_{2m-1}(B)\leq \sqrt{2}\sec  \theta \cdot  \lambda_{m} (\mathfrak{Re}  B),\;\;s_{2m}(B)\leq \sqrt{2}\sec  \theta \cdot  \lambda_{m} (\mathfrak{Re}  B), \;\;m =1,2,...\,  .
$$
This   gives us the upper estimate for the singular values of the operator $B.$
\end{proof}

However, to      obtain the lower estimate we need involve Lemma 3.1 \cite[p.336]{firstab_lit:kato1980}, Theorem 3.2 \cite[p.337]{firstab_lit:kato1980}.
 Consider an unbounded   operator $T,\, \Theta(T)\subset \mathfrak{L}_{0}(\theta),$  in accordance with the first representation theorem   \cite[p. 322]{firstab_lit:kato1980}, we can consider its Friedrichs extension the m-sectorial operator $W,$ in its own turn   due to the results \cite[p.337]{firstab_lit:kato1980}, it has a real part $H$ which coincides with the Hermitian  real component if we deal with a bounded operator. Note that by virtue of the sectorial property the operator $H$ is non-negotive. Further, we consider the case $\mathrm{N}(H)=0$ it follows that $\mathrm{N}(H^{\frac{1}{2}})  =0.$ To prove this fact we should note that $\mathrm{def}H=0,$ considering inner product with the element belonging to $\mathrm{N}(H^{\frac{1}{2}})$ we obtain easily the fact  that it must equal to zero.
Having analyzed the proof of Theorem 3.2 \cite[p.337]{firstab_lit:kato1980}, we see that its statement remains true in the modified  form even in the case if we lift the m-accretive condition, thus under the  sectorial condition imposed upon the closed densely defined  operator $T,$ we get the following inclusion
$$
T\subset H^{1/2}(I+iG)H^{1/2},
$$
here   symbol $G$ denotes    a bounded selfadjoint operator in $\mathfrak{H}.$ However, to obtain the asymptotic formula established in Theorem 5 \cite{firstab_lit(arXiv non-self)kukushkin2018} we cannot be satisfied by the made assumptions but require   the existence of the resolvent at the point zero and its compactness. In spite of the fact that   we can  proceed our narrative under the weakened conditions regarding the operator $W$ in comparison with H1,H2, we can   claim that the statement of  Theorem 5 \cite{firstab_lit(arXiv non-self)kukushkin2018} remains true under the assumptions made above, we prefer to deploy H1,H2 what guarantees the conditions we need and at the same time provides a description of the matter under the natural point of view.

\begin{lem}\label{L2} Assume that the hypotheses H1,H2  holds  for the operator $W,$ moreover
$$
 \|\mathfrak{Im}W /\mathfrak{Re}W\|_{2}< 1,
$$
then $$
  \lambda^{-1}_{2n} \left(\mathfrak{Re}W \right)\leq C s_{n}(R_{W}),\;n\in \mathbb{N}.
$$
\end{lem}
\begin{proof}
Firstly, let us show that $\mathrm{D}(W^{2})$ is a dense set in $\mathfrak{H}_{+}.$ Since the operator $W$ is closed and strictly accretive, we have $\mathrm{R}(W)=\mathfrak{H},$ hence there exists the preimage of the set $\mathfrak{M},$ let us denote it by $\mathfrak{M}'.$ Consider an arbitrary element $x_{0}\in \mathfrak{H}$ and denote its preimage by $x'_{0},$ we  have
$$
\|W(x'_{0}-x'_{n})\|_{\mathfrak{H}}\geq C\|x'_{0}-x'_{n}\|_{\mathfrak{H}_{+}},
$$
where $\{x_{n}\}_{1}^{\infty}\subset \mathfrak{M}'.$ Hence, the set $\mathfrak{M}'$ is dense in $\mathrm{D}(W)$  in  the sense of the norm $\mathfrak{H}_{+},$ hence it is dense in $\mathfrak{H}_{+}$ and consequently the set  $\mathrm{D}(W^{2})$   is dense  in $\mathfrak{H}_{+}$ since $\mathfrak{M}'\subset \mathrm{D}(W^{2}).$ Here, we should note that we have proved the  fulfilment of the condition H1 for the operator $W^{2}$ with respect to the same pair of Hilbert spaces.

Note that    under the   assumptions H1,H2, using the reasonings of Theorem 3.2 \cite{firstab_lit:kato1980}, we have the following representation
$$
W=H^{1/2}(I+iG)H^{1/2},\;W^{\ast}=H^{1/2}(I-iG)H^{1/2}.
$$
It follows easily from this formula that the Hermitian components of the operator $W$ are defined, we have
$
\mathfrak{Re}W =H,\;\;\mathfrak{Im}W=H^{1/2}GH^{1/2}.
$
Using the decomposition
$
W=\mathfrak{Re}W +i\mathfrak{Im}W,\;W^{\ast}=\mathfrak{Re}W -i\mathfrak{Im}W,
$
we get easily
$$
 \left(  \frac{W^{2}+W^{\ast\,2}}{2}\, f,f\right)=\left\|\mathfrak{Re}W f\right\|^{2}-\left\|\mathfrak{Im}W f\right\|^{2};\,
$$
$$
  \left(  \frac{W^{2}-W^{\ast\,2}}{2i}\, f,f\right) =  (\mathfrak{Im}W \,\mathfrak{Re}Wf,f) +(\mathfrak{Re}W\, \mathfrak{Im}Wf,),\;f\in \mathrm{D}(W^{2}).
$$
Using simple reasonings,  we can rewrite the above formulas in terms of Theorem 3.2 \cite{firstab_lit:kato1980}, we have
\begin{equation}\label{eq1}
\mathrm{Re}(W^{2}f,f)=\|Hf\|^{2}-\|H^{1/2}GH^{1/2}f\|^{2},\;\mathrm{Im} (W^{2}f,f)=\mathrm{Re}(H^{1/2}GH^{1/2}f,H f),
$$
$$
\,f\in \mathrm{D}(W^{2}).
\end{equation}
Consider a set of eigenvalues $\{\lambda_{n}\}_{1}^{\infty}$ and a complete system of orthonormal vectors $\{e_{n}\}_{1}^{\infty}$ of the operator $H,$ using the matrix form of the operator $G,$     we have
$$
\|Hf\|^{2}=\sum\limits_{n=1}^{\infty}  |\lambda _{n}|^{2}|f_{n}|^{2},\;\|H^{1/2}GH^{1/2}f\|^{2}=\sum\limits_{n=1}^{\infty} \lambda_{n} \left| \sum\limits_{k=1}^{\infty} b_{nk}\sqrt{\lambda_{k}}f_{k}\right|^{2},
$$
$$
\mathrm{Re}(H^{1/2}GH^{1/2}f,H^{1/2}f)=\mathrm{Re}\left(\sum\limits_{n=1}^{\infty} \lambda^{3/2}_{n}f_{n}  \sum\limits_{k=1}^{\infty} b_{nk}\sqrt{\lambda_{k}}\bar{f_{k}}\right),
$$
where $b_{nk}$  are the matrix coefficients of the operator $G.$
Applying the Cauchy-Swarcz inequality, we get
$$
 \|H^{1/2}GH^{1/2}f\|^{2}\leq\sum\limits_{n=1}^{\infty} \lambda_{n} \left|\sum\limits_{k=1}^{\infty} |\lambda_{k} f_{k}|^{2} \sum\limits_{k=1}^{\infty}
  |b_{ nk}|^{2} /\lambda_{k}\right|\leq \|Hf\|^{2} \sum\limits_{  n,k=1}^{\infty}|b_{ nk}|^{2}\lambda_{n}/\lambda_{k} ;
$$
$$
 |\mathrm{Re}(H^{1/2}GH^{1/2}f,H f)|\leq \|Hf\| \left(\sum\limits_{n=1}^{\infty}   \left|\sum\limits_{k=1}^{\infty}  \bar{b}_{ nk}   \sqrt{\lambda_{n}\lambda_{k}}f_{k}\right|^{2}
      \right)^{1/2} \leq\|Hf\|^{2} \left(\sum\limits_{ n,k=1}^{\infty}    |b_{nk}|^{2}\lambda_{n}/\lambda_{k} \right)^{1/2} .
$$
In accordance with the definition of the sectorial property, we require
$$
|\mathrm{Im}(W^{2}f,f)|\leq \tan\theta \cdot\mathrm{Re}(W^{2}f,f),\; 0<\theta<\pi/2.
$$
Therefore, the sufficient conditions of the  sectorial property can be expressed as follows
$$
   \|Hf\|^{2}\left(\sum\limits_{ n,k=1}^{\infty}    |b_{nk}|^{2}/\lambda_{k} \right)^{1/2}\!\!\!  \leq   \|Hf\|^{2} \left(1-\sum\limits_{  n,k=1}^{\infty}|b_{ nk}|^{2}\lambda_{n}/\lambda_{k}\right)\tan\theta   ;
 $$
 $$
 \sum\limits_{ k,n=1}^{\infty}|b_{nk}|^{2}\lambda_{n}/\lambda_{k}  +  \mathrm{ctg} \theta\left(\sum\limits_{ k,n=1}^{\infty}   |b_{nk}|^{2}\lambda_{n}/\lambda_{k} \right)^{1/2}\leq 1,
$$
where $\theta$ is the  semi-angle of the supposed sector.
Solving the corresponding quadratic equation, we obtain the desired estimate
\begin{equation}\label{eq6d}
 \left(\sum\limits_{ k,n=1}^{\infty} |b_{nk}|^{2}\lambda_{n}/\lambda_{k}\right)^{1/2}<\frac{1}{2}\left\{\sqrt{  \mathrm{ctg}^{2} \theta  +4}-  \mathrm{ctg} \theta \right\} .
\end{equation}
 Having noticed the fact that    the right hand side of \eqref{eq6d} tends to one from bellow when $\theta$ tends to $\pi/2,$ we obtain the condition of the sectorial property expressed in terms of  the absolute norm
 \begin{equation}\label{eq7c}
  \|H^{1/2}GH^{-1/2}\| _{2}:=\left(\sum\limits_{ k,n=1}^{\infty} |b_{nk}|^{2}\lambda_{n}/\lambda_{k}\right)^{1/2}<1,
 \end{equation}
  in this case, we we can choose the semi-angle of the sector using the following relation
  $$
   \tan\theta = \frac{N}{1-N^{2}}+\varepsilon,\;N:=\|H^{1/2}GH^{-1/2}\| _{2},
  $$
  where $\varepsilon$ is an arbitrary small positive number. Thus, we can resume that in the value of the absolute norm less than one than the operator $W^{2}$ is sectorial and the value of the absolute norm defines the semi-angle.
  Note that coefficients  $  b_{nk} \sqrt{\lambda_{n}/\lambda_{k}},\;\overline{b_{kn}} \sqrt{\lambda_{n}/\lambda_{k}} $ correspond to the matrices of the operators respectively
$$
H^{1/2}GH^{-1/2}f=\sum\limits_{n=1}^{\infty}  \lambda^{1/2}_{n} e_{n}  \sum\limits_{k=1}^{\infty} b_{nk}\lambda^{-1/2}_{k} f_{k},\;H^{-1/2}GH^{1/2}f=\sum\limits_{n=1}^{\infty}  \lambda^{-1/2}_{n} e_{n}  \sum\limits_{k=1}^{\infty} b_{nk}\lambda^{1/2}_{k} f_{k}   .
$$
Thus, if the absolute operator norm exists, i.e.
$$
\|H^{1/2}GH^{-1/2}\| _{2} <\infty,
$$
then both of them belong to the so-called Hilbert-Schmidt class simultaneously, but it is clear without involving the   absolute norm  since the above operators are adjoint.   It is remarkable that,   we can write formally the obtained estimate in terms of the Hermitian components of the operator, i.e.
$$
\|\mathfrak{Im}W /\mathfrak{Re}W\|_{2}< 1.
$$
Bellow, for a convenient form of writing, we will use a short-hand notation $A:=R_{ W },$ where it is necessary. The next step is to establish the asymptotic formula
\begin{equation}\label{eq8c}
\lambda_{n}\left( \frac{A^{2}+A^{ 2\ast}}{2}\right)\asymp \lambda^{-1}_{n} \left(\mathfrak{Re}W^{2} \right),\;n\rightarrow\infty.
\end{equation}
However, we cannot apply directly  Theorem 5 \cite{firstab_lit(arXiv non-self)kukushkin2018} to the operator $W^{2},$  thus we   are compelled to modify the proof having taken into account weaker conditions and the additional condition \eqref{eq7c}.

 Let us observe that the compactness of  the operator $R_{W}(\lambda),\,\lambda \in \mathrm{P}(W)$ gives us the compactness of the operator $W^{-2}.$ Since the latter is sectorial, it follows easily that  $R_{W^{2}}(\lambda),\,\lambda \in \mathrm{P}(W^{2})$ is compact, since the outside of the sector belongs to the resolvent set and the resolvent compact at least at one point is compact everywhere on the resolvent set.  Note that due to the reasonings given above the following relation holds
\begin{equation}\label{7x}
\mathrm{Re}(W^{2}f,f) \geq C \|Hf\|^{2} \geq C\|f\|^{2}_{\mathfrak{H}_{+}},\,f\in \mathrm{D}(W^{2}),
\end{equation}
the latter inequality can be obtained easily (see (28) \cite{firstab_lit(arXiv non-self)kukushkin2018}). Thus, we obtain the fact that the operator $W^{2}$ is sectorial, strictly accretive operator, hence falls in the scope of the first representation theorem in accordance with which there exists one to one correspondence between the closed densely defined sectorial forms and m-sectorial operators. Using this fact, we can claim that the real part $H_{1}:=\mathrm{Re} W^{2}$ is defined and the following relations hold in accordance with the representation theorem i.e., Theorem 3.2 \cite[p.337]{firstab_lit:kato1980}, we get
$$
W^{2}=H_{1}^{1/2}(I+iG_{1})H_{1}^{1/2},\;W^{ 2\ast}=H_{1}^{1/2}(I+iG_{2})H_{1}^{1/2},
$$
where $G_{1},G_{2}$ are selfadjoint bounded operators. Now by direct calculation, we can verify   that $H_{1}=\mathfrak{Re}W^{2},$ we should also note that $\mathrm{D}(W^{2})$ is a core of  the corresponding closed densely defined sectorial form $\mathfrak{h}$ put in correspondence to the operator $H_{1}$ by the first representation theorem, i.e. $\mathrm{D}_{0}(\mathfrak{h})=\mathrm{D}(W^{2}).$
Let us show that $G_{1}=-G_{2}.$ We have
\begin{equation*}
  H_{1}f\! =\!\frac{1}{2}\left[H_{1}^{\frac{1}{2}}(I+i G_{1})  +H_{1}^{\frac{1}{2}}(I+i G_{2})\right]H_{1}^{\frac{1}{2}}\!  =
  $$
  $$
  = \! H_{1} f +
 \frac{i}{2} H_{1}^{\frac{1}{2}}\left(G_{1}+G_{2}\right)  H_{1}^{\frac{1}{2}}f  ,\;f\in \mathfrak{M}'.
\end{equation*}
By virtue of  inequality \eqref{7x}, we see that the operator $H_{1}$ is strictly accretive, therefore
  $\mathrm{N}(H_{1})=0 ;\;( G_{1}+G_{2})H_{1}^{1/2}=0.$ Since
   $$
   \mathfrak{H}=\overline{\mathrm{R}(H_{1}^{1/2})}\oplus \mathrm{N}(H_{1}^{1/2}),
   $$
   then $G_{1}=G_{2}=:G.$ Applying the reasonings represented in Theorem 5 \cite{firstab_lit(arXiv non-self)kukushkin2018}, we obtain the fact that    $H_{1}^{-1/2}$ is a bounded operator  defined on $\mathfrak{H}.$
   Using the  properties  of the operator $G,$    we get
$\|(I+ iG)f\| \cdot\|f\| \geq\mathrm{Re }\left([I+ iG]f,f\right)  =\|f\|^{2} ,\,f\in \mathfrak{H}.$ Hence
$
\|(I+ iG)f\|  \geq \|f\| ,\,f\in \mathfrak{H}.
$
 It implies that the operators  $I+ iG$ are invertible. The reasonings   corresponding to the operator $I-iG$ are absolutely analogous.
   Therefore
  \begin{equation}\label{3.19}
A^{2}=H_{1}^{-\frac{1 }{2}}(I+iG  )^{-1} H_{1}^{- \frac{1}{2}},\;A^{2\ast}=H_{1}^{-\frac{1 }{2}}(I-iG  )^{-1} H_{1}^{- \frac{1}{2}}.
\end{equation}
Using simple calculation based upon the operator properties established above, we get
\begin{equation}\label{3.20}
\mathfrak{Re} A^{2} =\frac{1}{2}\,H_{1}^{-\frac{1 }{2}}  (I+G^{2} )^{-1}  H_{1}^{- \frac{1}{2}}.
\end{equation}
Therefore
 $$
\left(\mathfrak{Re} A^{2}  f,f\right) =\left(H_{1}^{-\frac{1 }{2}} (I+G^{2} )^{-1}     H_{1}^{- \frac{1}{2}}   f,f\right)  \leq
  \|(I+G^{2} )^{-1}      \|\cdot\left(R_{H_{1}}  f,f\right) ,\;f\in \mathfrak{H}.
$$
On the other hand,  it is easy to see that  $ ((I+G^{2} )^{-1}f,f) \geq \|(I+G^{2} )^{-1}f\|^{2} .$ At the same time   it is obvious that   $ I+G^{2} $ is bounded and we have   $\|(I+G^{2} )^{-1}f\| \geq \| I+G^{2}   \|^{-1} \|f\| .$ Using   these estimates, we have
$$
\left(\mathfrak{Re} A^{2}  f,f\right) =\left( (I+G^{2} )^{-1}     H_{1}^{- \frac{1}{2}}   f,H_{1}^{-\frac{1 }{2}}f\right) \geq
\|(I+G^{2} )^{-1}     H_{1}^{- \frac{1}{2}}   f \|^{2} \geq
$$
$$
\geq  \|  I+G^{2}    \|^{-2} \cdot   \left(R_{H_{1} }  f,f\right)_{\mathfrak{H}},\;f\in \mathfrak{H}.
$$
Using relation \eqref{7x}, we obtain the fact that the resolvent  $R_{H_{1}}$ is compact, the fact that $\mathfrak{Re} A^{2}$ is compact is obvious.
 Thus, analogously to the reasonings of  Theorems 5 \cite{firstab_lit(arXiv non-self)kukushkin2018}  applying the minimax principle we obtain the desired asymptotic formula \eqref{eq8c}.
 Further, we will use  the following formula obtained due to the  positiveness of the  squared Hermitian imaginary  component of the operator $A,$ we have
$$
\frac{A^{2}+A^{  2\ast}}{2}=\frac{A^{2}+A^{\ast 2}}{2}\leq A^{\ast}A+AA^{\ast}.
$$
Applying the   corollary of the well-known Allakhverdiyev theorem   (Ky Fan  \cite{firstab_lit:Fan}),   see   Corollary  2.2 \cite{firstab_lit:1Gohberg1965} (Chapter II, $\S$ 2.3),   we have
$$
\lambda_{2n}\left(  A^{\ast}A+AA^{\ast}\right)\leq \lambda_{n}(A^{\ast}A)+\lambda_{n}(AA^{\ast}),\; n\in \mathbb{N}.
$$
Taking into account the fact $s_{n}(A)=s_{n}(A^{\ast}),$ using the minimax principle, we obtain the estimate
$$
s^{2}_{n}(A)\geq C \lambda_{2n}\left( \frac{A+A^{  2\ast}}{2}\right),\;n\in \mathbb{N},
$$
applying \eqref{eq8c}, we obtain
$$
s^{2}_{n}(A)\geq C \lambda^{-1}_{2n} \left(\mathfrak{Re}W^{2} \right),\;n\in \mathbb{N}.
$$
Here, it is rather reasonable to apply formula \eqref{eq1} which  gives us
$$
\|f\|^{2}\leq \|f\|^{2}_{\mathfrak{H}_{+}} \leq\left(\mathfrak{Re}W^{2}f,f\right)\leq\left(H f, Hf\right) ,\;f\in \mathrm{D}(W^{2}),
$$
what in its own turn, collaboratively with the minimax principle leads  us to the theorem statement.
\end{proof}
\begin{remark} It is remarkable that the central point of  the above proof is the representation theorems, in accordance with the first one we have a plain  construction of the operator real part equaling the Hermitian real component. These allow  us to implement the simplified scheme of reasonings represented in  \cite{firstab_lit(arXiv non-self)kukushkin2018}.
\end{remark}
The results given above deserve to be formulated in the following stylistically convenient form
\begin{teo}\label{T1}
Assume that the hypotheses H1,H2 hold  for the operator $W,$ moreover  $$
\|\mathfrak{Im}W /\mathfrak{Re}W\|_{2}< 1,
$$
then
$$
s_{n}(R_{W})\asymp   \lambda^{-1}_{n} \left(\mathfrak{Re} W \right).
$$
\end{teo}
\begin{proof}
The proof follows from Lemmas \ref{L1},\ref{L2}.
\end{proof}

\section{Mathematical applications}

{\bf The low bound for the Schatten index  of the  perturbed   differential  operator}\\

\noindent{\bf 1.} Trying to show application of Lemma \ref{L1}, we produce  an example  of a non-selfadjoin operator that is not completely subordinated in the sense of forms (see \cite{firstab_lit:Shkalikov A.}, \cite{firstab_lit(arXiv non-self)kukushkin2018}). The pointed out fact means that, we cannot deal with the operator applying methods \cite{firstab_lit:Shkalikov A.} for they do not work.
   Consider a  differential operator acting in the complex Sobolev  space
$$
\mathcal{L}f := (c_{k}f^{(k)})^{(k)} + (c_{k-1}f^{(k-1)})^{(k-1)}+...+  c_{0}f,
$$
$$
\mathrm{D}(\mathcal{L}) = H^{2k}(I)\cap H_{0}^{k}(I),\,k\in \mathbb{N},
$$
where    $I: = (a, b) \subset \mathbb{R},$ the   complex-valued coefficients
$c_{j}(x)\in C^{(j)}(\bar{I})$ satisfy the condition $  {\rm sign} (\mathrm{Re} c_{j}) = (-1)^{j} ,\, j = 1, 2, ..., k.$
 Consider  a linear combination of   the  Riemann-Liouville  fractional differential   operators
  (see \cite[p.44]{firstab_lit:samko1987})   with the  constant  real-valued  coefficients
$$
\mathcal{D}f:=p_{n}D_{a+}^{\alpha_{n}}+q_{n}D_{b-}^{\beta_{n}}+p_{n-1}D_{a+}^{\alpha_{n-1}}+q_{n-1}D_{b-}^{\beta_{n-1}}+...+
p_{0}D_{a+}^{\alpha_{0}}+q_{0}D_{b-}^{\beta_{0}},
$$
$$
\mathrm{D}(\mathcal{D}) = H^{2k}(I)\cap H_{0}^{k}(I),\,n\in \mathbb{N},
$$
where $\alpha_{j},\beta_{j}\geq 0,\,0 \leq [\alpha_{j}],[\beta_{j}] < k,\, j = 0, 1, ..., n.,\;$
\begin{equation*}
 q_{j}\geq0,\;{\rm sign}\,p_{j}= \left\{ \begin{aligned}
  (-1)^{\frac{[\alpha_{j}]+1}{2}},\,[\alpha_{j}]=2m-1,\,m\in \mathbb{N},\\
\!\!\!\!\!\!\! \!\!\!\!(-1)^{\frac{[\alpha_{j}]}{2}},\;\,[\alpha_{j}]=2m,\,\,m\in \mathbb{N}_{0}   .\\
\end{aligned}
\right.
\end{equation*}
The following result is represented in the paper  \cite{firstab_lit(arXiv non-self)kukushkin2018},    consider the operator
$$
G=\mathcal{L}+\mathcal{D},
$$
$$
\mathrm{D}(G)=H^{2k}(I)\cap H_{0}^{k}(I).
$$
It is  clear that it is an operator with a compact resolvent, however for the accuracy  we will prove this fact, moreover we will produce a pair of Hilbert spaces  so that conditions H1,H2 holds. It follows that the resolvent is compact, thus we can observe the problem of calculating Schatten index. Apparently, it may happen that the direct calculation  of the  singular numbers or  their estimation is rather complicated   since we have the following construction,
$$
GG^{\ast}\supset  (\mathcal{L}+\mathcal{D})(\mathcal{L}^{\ast}+\mathcal{D}^{\ast})\supset  \mathcal{L}\mathcal{L}^{\ast}+ \mathcal{D}\mathcal{L}^{\ast}+\mathcal{L}\mathcal{D}^{\ast}+\mathcal{D}\mathcal{D}^{\ast}
$$
where inclusions must satisfy some conditions connected with the core of the operator form for in other case we have a risk to lose some singular numbers. In spite of the fact that    the shown difficulties in many cases  can be  eliminated the offered method of singular numbers estimation becomes apparently relevant.

Let us prove the fulfilment of the conditions H1,H2 under the assumptions  $\mathfrak{H} := L_{2}(I),\, \mathfrak{H}^{+} := H_{0}^{k}(I),\,\mathfrak{M}:=C_{0}^{\infty}(I).$
The fulfillment of the condition H1 is obvious, let us show the fulfilment of the condition  H2. It is easy
to see that
$$
\mathrm{Re}(\mathcal{L}f,f)_{L_{2}(I)}\geq\sum\limits_{j=0}^{k}|\mathrm{Re} c_{j}|\,\|f^{(j)}\|^{2}_{L_{2}(I)}\geq C \|f^{(j)}\|^{2}_{H_{0}^{k}(I)},\;f\in \mathrm{D}(\mathcal{L}).
$$
On the other hand
$$
|(\mathcal{L}f,f)_{L_{2}(I)}|=\left|\sum\limits_{j=0}^{k}(-1)^{j}(c_{j}f^{(j)},g^{(j)} )_{L_{2}(I)}\right|\leq
\sum\limits_{j=0}^{k}\left|(c_{j}f^{(j)},g^{(j)} )_{L_{2}(I)}\right|\leq
$$
$$
\leq C \sum\limits_{j=0}^{k} \|f^{(j)}\| _{L_{2}(I)}\|g^{(j)}\| _{L_{2}(I)}\leq
\|f \| _{H^{k}_{0}(I)}\|g \| _{H^{k}_{0}(I)},\;f\in \mathrm{D}(\mathcal{L}).
$$
Consider the  Riemann-Liouville   operators of fractional differentiation of   arbitrary non-negative
order $\alpha$ (see \cite[p.44]{firstab_lit:samko1987})  defined by the expressions
\begin{equation*}
 D_{a+}^{\alpha}f=\left(\frac{d}{dx}\right)^{[\alpha]+1}\!\!\!\!I_{a+}^{1-\{\alpha\}}f;\;
 D_{b-}^{\alpha}f=\left(-\frac{d}{dx}\right)^{[\alpha]+1}\!\!\!\!I_{b-}^{1-\{\alpha\}}f,
\end{equation*}
where the fractional integrals of      arbitrary positive order  $\alpha$ defined by
$$
\left(I_{a+}^{\alpha}f\right)\!(x)=\frac{1}{\Gamma(\alpha)}\int\limits_{a}^{x}\frac{f(t)}{(x-t)^{1-\alpha}}dt,
 \left(I_{b-}^{\alpha}f\right)\!(x)=\frac{1}{\Gamma(\alpha)}\int\limits_{x}^{b}\frac{f(t)}{(t-x)^{1-\alpha}}dt
, f\in L_{1}(I).
$$
Suppose  $0<\alpha<1,\, f\in AC^{l+1}(\bar{I}),\,f^{(j)}(a)=f^{(j)}(b)=0,\,j=0,1,...,l;$ then the next formulas follows
from   Theorem 2.2 \cite[p.46]{firstab_lit:samko1987}
\begin{equation}\label{4.4}
 D_{a+}^{\alpha+l}f= I_{a+}^{1- \alpha }f^{(l+1)},\;
 D_{b-}^{\alpha+l}f= (-1)^{l+1}I_{b-}^{1- \alpha }f^{(l+1)}.
\end{equation}
  Further, we need  the following inequalities    (see  \cite{firstab_lit:1kukushkin2018})
\begin{equation}\label{4.5}
\mathrm{Re} (D_{a+}^{\alpha}f,f)_{L_{2}(I)}\geq C\|f\|^{2}_{L_{2}(I)},\,f\in I_{a+}^{\alpha}(L_{2}),\;
$$
$$
\mathrm{Re} (D_{b-}^{\alpha}f,f)_{L_{2}(I)}\geq C\|f\|^{2}_{L_{2}(I)},\,f\in I_{b-}^{\alpha}(L_{2}),
\end{equation}
where $I_{a+}^{\alpha}(L_{2}),I_{b-}^{\alpha}(L_{2})$ are the  classes of  the  functions representable by the fractional integrals (see\cite{firstab_lit:samko1987}).
  Consider the following operator  with the  constant  real-valued  coefficients
$$
\mathcal{D}f:=p_{n}D_{a+}^{\alpha_{n}}+q_{n}D_{b-}^{\beta_{n}}+p_{n-1}D_{a+}^{\alpha_{n-1}}+q_{n-1}D_{b-}^{\beta_{n-1}}+...+
p_{0}D_{a+}^{\alpha_{0}}+q_{0}D_{b-}^{\beta_{0}},
$$
$$
\mathrm{D}(\mathcal{D}) = H^{2k}(I)\cap H_{0}^{k}(I),\,n\in \mathbb{N},
$$
where $\alpha_{j},\beta_{j}\geq 0,\,0 \leq [\alpha_{j}],[\beta_{j}] < k,\, j = 0, 1, ..., n.,\;$
\begin{equation*}
 q_{j}\geq0,\;{\rm sign}\,p_{j}= \left\{ \begin{aligned}
  (-1)^{\frac{[\alpha_{j}]+1}{2}},\,[\alpha_{j}]=2m-1,\,m\in \mathbb{N},\\
\!\!\!\!\!\!\! \!\!\!\!(-1)^{\frac{[\alpha_{j}]}{2}},\;\,[\alpha_{j}]=2m,\,\,m\in \mathbb{N}_{0}   .\\
\end{aligned}
\right.
\end{equation*}
Using \eqref{4.4},\eqref{4.5},  we get
$$
(p_{j}D_{a+}^{\alpha_{j}}f,\!f)_{L_{2}(I)}\!=
\!p_{j}\!\left(\!\!\left(\frac{d}{dx}\right)^{\!\!\!m}\!\!D_{a+}^{m-1+\{\alpha_{j}\}}\!\!f,\!f   \!\right)_{\!\!L_{2}(I)}\!\!\!\!\!\! =
(\!-1)^{m}p_{j}\!\left(\! I_{a+}^{1-\{\alpha_{j}\}}\!\!f^{(m)}\!\!,\!f^{(m)}   \!\right)_{\!\!L_{2}(I)}\!\! \geq
$$
$$
\geq C\left\|I_{a+}^{1-\{\alpha_{j}\}}f^{(m)}\right\|^{2}_{L_{2}(I)}=
C\left\|D_{a+}^{\{\alpha_{j}\}}f^{(m-1)}\right\|^{2}_{L_{2}(I)}\geq C \left\| f^{(m-1)}\right\|^{2}_{L_{2}(I)},
$$
 where  $f \in \mathrm{D}(\mathcal{D})$ is    a real-valued function and   $ [\alpha_{j}]=2m-1,\,m\in \mathbb{N}.$
  Similarly,  we obtain for orders  $ [\alpha_{j}]=2m,\,m\in \mathbb{N}_{0}$
$$
(p_{j}D_{a+}^{\alpha_{j}}f,f)_{L_{2}(I)}=p_{j}\left( D_{a+}^{2m +\{\alpha_{j}\}}f,f   \right)_{L_{2}(I)}=(-1)^{m}p_{j}\left( D_{a+}^{m+\{\alpha_{j}\}}f ,f^{(m)}   \right)_{L_{2}(I)}=
$$
$$
=(-1)^{m}p_{j}\left( D_{a+}^{ \{\alpha_{j}\}}f^{(m)} ,f^{(m)}   \right)_{\!L_{2}(I)}\geq C \left\| f^{(m)}\right\|^{2}_{L_{2}(I)}.
$$
Thus in both cases,  we have
$$
(p_{j}D_{a+}^{\alpha_{j}}f,f)_{L_{2}(I)}\geq C \left\| f^{(s)}\right\|^{2}_{L_{2}(I)},\;s= \big[[\alpha_{j}]/2\big] .
$$
 In the same way, we obtain the inequality
$$
(q_{j}D_{b-}^{\alpha_{j}}f,f)_{L_{2}(I)}\geq C \left\| f^{(s)}\right\|^{2}_{L_{2}(I)},\;s= \big[[\alpha_{j}]/2\big] .
$$
  Hence in the
complex case we have
$$
\mathrm{Re}(\mathcal{D}f,f)_{L_{2}(I)}\geq C \left\| f \right\|^{2}_{L_{2}(I)},\;f\in \mathrm{D}(\mathcal{D}).
$$
Combining   Theorem 2.6 \cite[p.53]{firstab_lit:samko1987}  with  \eqref{4.4}, we get
$$
\left\| p_{j}D_{a+}^{\alpha_{j}}f \right\| _{L_{2}(I)}=  \left\|  I_{a+}^{1-\{\alpha_{j}\}}f^{([\alpha_{j}]+1)} \right\| _{L_{2}(I)}
 \leq C   \left\|   f^{([\alpha_{j}]+1)} \right\|_{L_{2}(I)}\leq C   \left\|   f  \right\|_{H^{k}_{0}(I)};
 $$
 $$
 \;\left\|q_{j}D_{b-}^{\alpha_{j}}f \right\| _{L_{2}(I)}
\leq  C   \left\|   f  \right\|_{H^{k}_{0}(I)},\;f\in \mathrm{D}(\mathcal{D}).
$$
 Hence, we obtain
$$
\left\| \mathcal{D}f \right\| _{L_{2}(I)}\leq C \left\|  f \right\|_{H^{k}_{0}(I)},\;f\in \mathrm{D}(\mathcal{D}).
$$
Taking into account the relation
$$
\left\|f\right\| _{L_{2}(I)}\leq C \left\|f\right\| _{H^{k}_{0}(I)},\,f\in H^{k}_{0}(I),
$$
Combining the above estimates, we get
$$
\mathrm{Re} (Gf,f)_{L_{2}(I)}\geq C\|f\|^{2}_{H^{k}_{0}(I)},\; |(Gf,g)_{L_{2}(I)}|\leq \|f \| _{H^{k}_{0}(I)}\|g \| _{H^{k}_{0}(I)},\,f,g\in C^{\infty}_{0}(I).
$$
Thus, we have obtained the desired result.

To deploy  the  minimax principle for  eigenvalues estimating, we come to the following relation
$$
C_{1}\|f \|^{2} _{H^{k}_{0}(I)}\leq(\mathfrak{Re} G f,f)_{L_{2}(I)}\leq C_{2}\|f \|^{2} _{H^{k}_{0}(I)},
$$
from what follows  easily, due to the asymptotic formulas for  the selfadjoint operators  eigenvalues (see \cite{firstab_lit:Rosenblum}), the fact
$$
\lambda_{n}(\mathfrak{Re} G)\asymp n^{2k},\;n\in  \mathbb{N},
$$
therefore applying Lemma \ref{L1} collaboratively with the asymptotic equivalence formula (see Theorem 5 \cite{firstab_lit(arXiv non-self)kukushkin2018})
$$
\lambda^{-1}_{n}(\mathfrak{Re} G)\asymp \lambda_{n}(\mathfrak{Re}R_{ G}),\;n\in  \mathbb{N},
$$
 we obtain the fact
$$
R_{G}\in \mathfrak{S}_{p},\,\inf p\leq 1/2k.
$$
Thus, it gives us an opportunity to establish the range of  the Schatten index.\\

\noindent{\bf 2.} Let us show the application of Lemma \ref{L2}, firstly consider the following reasonings
$$
\| \mathfrak{Im}W H^{-1} \|_{2}=\| H^{-1}\mathfrak{Im}W\|_{2}=\sum\limits_{n,k=1}^{\infty}|(\mathfrak{Im} We_{k},H^{-1}e_{n})|^{2}=
\sum\limits_{n,k=1}^{\infty}\lambda^{-2}_{n}|(e_{k}, \mathfrak{Im} W e_{n})|^{2}=
$$
$$
=\sum\limits_{n=1}^{\infty}\lambda^{-2}_{n}||  \mathfrak{Im} W e_{n} ||^{2},
$$
where $\{e_{n}\}_{1}^{\infty}$ is the orthonormal  set of the eigenvectors of the operator $H.$ Thus, we obtain the following condition
\begin{equation}\label{12x}
\sum\limits_{n=1}^{\infty}\lambda^{-2}_{n}||  \mathfrak{Im} W e_{n} ||^{2}<1,
\end{equation}
which guaranties the fulfilment of the additional   condition regarding  the absolute norm  in Lemma \ref{L2}. It is remarkable that this form of the condition is quiet convenient if we consider perturbations of differential operators.
   Bellow we observe  a simplified case of the operator considered in the previous paragraph. Consider
  $$
 Lf := - f''  + \xi D_{0+}^{\alpha }f ,\;
\mathrm{D}(L) = H^{2 }(I)\cap H_{0}^{1}(I) ,\;I=(0,\pi),\,\alpha\in (0,1/2),
$$
then
$$
C_{0}( L_{1}f,f)\leq(\mathfrak{Re}Lf,f)\leq C_{1}( L_{1}f,f),\; L_{1}f:=-f'',\;\mathrm{D}(L_{1})=\mathrm{D}(L).
$$
It is well-known fact that
$$
\lambda_{n}(L_{1})=n^{2},\;e_{n}=\cos nx.
$$
It is also clear that
$$
\mathfrak{Im }L \supset \xi(D_{0+}^{\alpha }-D_{\pi-}^{\alpha })/2i.
$$
In accordance with the first representation theorem $H^{2 }(I)\cap H_{0}^{1}(I)$  is a core of the form corresponding to the operator $L^{\ast},$ hence
$$
\mathfrak{Im }L = \xi(D_{0+}^{\alpha }-D_{\pi-}^{\alpha })/2i.
$$
Note that
$$
\left(D_{0+}^{\alpha }e_{n}\right)(x)=\frac{n}{\Gamma(1-\alpha)}\int\limits_{0}^{x}(x-t)^{-\alpha}\cos n t\, dt
$$
Applying the generalized Minkovskii inequality, we get
$$
\left(\int\limits_{0}^{\pi}\left|(D_{a+}^{\alpha }e_{n})(x)\right|^{2} dx\right)^{1/2}= \frac{n }{\Gamma(1-\alpha)}\left(\int\limits_{0}^{\pi}\left|\int\limits_{0}^{x}(x-t)^{ -\alpha}\cos n t\, dt\right|^{2}\right)^{1/2}\leq
$$
$$
\leq  \frac{n }{\Gamma(1-\alpha)}\int\limits_{0}^{\pi}\cos n t\, dt \left(\int\limits_{t}^{\pi}(x-t)^{ -2\alpha}d x \right)^{1/2}=
 \frac{n }{\sqrt{(1-2\alpha)}\Gamma(1-\alpha)}\int\limits_{0}^{\pi}(\pi-t)^{1/2- \alpha}\cos n t\, dt\leq
$$
$$
\leq
 \frac{n \pi^{1/2-\alpha}}{\sqrt{(1-2\alpha)}\Gamma(1-\alpha)}.
$$

Analogously, we obtain
$$
\left(\int\limits_{0}^{\pi}\left|(D_{\pi-}^{\alpha }e_{n})(x)\right|^{2} dx\right)^{1/2}\leq  \frac{n \pi^{1/2-\alpha}}{\sqrt{(1-2\alpha)}\Gamma(1-\alpha)}.
$$
Hence
$$
\|\mathfrak{Im }L e_{n}\|\leq \frac{n \xi\pi^{1/2-\alpha}}{\sqrt{(1-2\alpha)}\Gamma(1-\alpha)}.
$$
Therefore
$$
\sum\limits_{n=1}^{\infty}\lambda^{-2}_{n}(\mathfrak{Re}L)||  \mathfrak{Im} Le_{n} ||^{2}< \frac{  \xi^{2}\pi^{1 -2\alpha}}{ (1-2\alpha) \Gamma^{2}(1-\alpha)}\sum\limits_{n=1}^{\infty}\frac{1}{n^{2}}=\frac{  \xi^{2}\pi^{3 -2\alpha } }{6 (1-2\alpha) \Gamma^{2}(1-\alpha)}.
$$
Using this relation we can obviously impose a condition on $\xi$ that guarantees the fulfilment of the relation \eqref{12x}, i.e.
$$
 \xi  <\frac{    \sqrt{6 (1-2\alpha)} \Gamma (1-\alpha)   }{ \pi^{3/2 - \alpha }}.
$$
In accordance with Theorem \ref{T1}, the latter condition follows that
$$
s_{n}^{-1}(R_{L})\asymp n^{2},\;R_{L}\in \mathfrak{S}_{p},\,\inf p=1/2.
$$

\noindent{\bf Existence and uniqueness theorems for     Evolution equations via obtained results }\\

     We  consider   applications  to the differential equations in the concrete Hilbert spaces and involve such operators as Riemann-Liouville operator, Kipriyanov operator,    Riesz potential, difference operator.  Moreover, we produce  the artificially constructed normal operator for which the clarification of the Lidskii   results  relevantly works.
     Further, we will consider a Hilbert space $\mathfrak{H}$ which consists of   element-functions $u:\mathbb{R}_{+}\rightarrow \mathfrak{H},\,u:=u(t),\,t\geq0$    and we will assume that if $u$ belongs to $\mathfrak{H}$    then the fact  holds for all values of the variable $t.$ Notice that under such an assumption all standard topological  properties as completeness, compactness   etc.  remain correctly defined. We understand such operations as differentiation and integration in the generalized sense that is caused by the topology of the Hilbert space $\mathfrak{H},$  more detailed information can be found in the   Chapter 4   Krasnoselskii M.A.   Consider an arbitrary compact operator $B,$ we can form  the  operators  corresponding to the groups of eigenvalues i.e.
$$
\mathcal{P}_{\nu}(B,\alpha,t) \Leftrightarrow \lambda_{N_{\nu}+1},\lambda_{N_{\nu}+2},...,\lambda_{N_{\nu+1}},
$$
where $\{N_{\nu}\}_{0}^{\infty}$ is a sequence of natural numbers,
$$
\mathcal{P}_{\nu}(B,\alpha,t)= \frac{1}{2 \pi i}\int\limits_{\vartheta_{\nu}(B)}e^{- \lambda^{\alpha}  t}B(I-\lambda B)^{-1} d\lambda,\,\alpha>0,
$$
  $\vartheta_{\nu}(B)$ is a contour on the complex plain containing only the  eigenvalues $\lambda_{N_{\nu}+1},\lambda_{N_{\nu}+2},...,\lambda_{N_{\nu+1}}$ and no  more eigenvalues.

  The root vectors of the operator $B$ is called by the Abell-Lidskii basis if
 $$
  \sum_{\nu=0}^{\infty}\mathcal{P}_{\nu}(B,\alpha,t)\rightarrow I,\,t\rightarrow 0,
$$
where convergence is understood as   operator pointwise convergence   in the Hilbert space.  We can compare this definition  with the unit decomposition - the main principle of the spectral theorem.
We put the following contour   in correspondence to the operator
\begin{equation*}
\vartheta(B):=\left\{\lambda:\;|\lambda|=r>0,\,|\mathrm{arg} \lambda|\leq \theta+\varepsilon\right\}
 \cup\left\{\lambda:\;|\lambda|>r,\; |\mathrm{arg} \lambda|=\theta+\varepsilon\right\}.
\end{equation*}
 Consider the following hypotheses \\

\noindent  {\bf(S1)}  {\it Under the assumptions $B\in \mathfrak{S} _{p},\,\inf p\leq\alpha,\,\Theta(B) \subset   \mathfrak{L}_{0}(\theta),$  a sequence of natural numbers $\{N_{\nu}\}_{0}^{\infty}$ can be chosen so that
 \begin{equation*}\label{eq6a}
 \frac{1}{2\pi i}\int\limits_{\vartheta(B)}e^{-\lambda^{\alpha}t}B(I-\lambda B)^{-1}f d \lambda =\sum\limits_{\nu=0}^{\infty} \mathcal{P}_{\nu}(B,\alpha,t)f,
 \end{equation*}
 the latter series is absolutely convergent in the sense of the norm.}

  Combining the generalized integro-differential    operations, we can consider a  fractional differential operator
 in the Riemann-Liouvile sense,   i.e. in the formal form, we have
$$
   \mathfrak{D}^{1/\alpha}_{-}f(t):=-\frac{1}{\Gamma(1-1/\alpha)}\frac{d}{d t}\int\limits_{0}^{\infty}f(t+x)x^{-1/\alpha}dx,\;\alpha>1.
$$
Let us study   a Cauchy problem
\begin{equation}\label{17n}
   \mathfrak{D}^{1/\alpha}_{-}  u=  W  u ,\;u(0)=h\in \mathrm{D}(W ).
\end{equation}
Note that it is possible to apply  the Abell-Lidskii concept using the methods  \cite{firstab_lit:1Lidskii}, \cite{kukushkin2019}, \cite{kukushkin2021a}, \cite{firstab_lit:1kukushkin2021}, \cite{firstab_lit:2kukushkin2022}, \cite{firstab_lit(axi2022)} in the   case $R_{W}\in \mathfrak{S}_{p},\, \inf p \leq\alpha.$  We can resume the central result of the listed above papers is to find conditions under which the hypotheses  S1  holds.\\

{\it The latter follows that
 there exists a solution of the Cauchy problem \eqref{17n} in the form
$$
u(t)=\sum\limits_{\nu=0}^{\infty} \mathcal{P}_{\nu}(B,\alpha,t)h.
$$
Apparently, under this point of view the results of the paper become relevant since we can find the exact value of the Schatten index in accordance with Theorem \ref{T1}.}\\

To demonstrate the claimed result we produce an example dealing with   well-known operators. Consider a rectangular domain in the space $\mathbb{R}^{n},\;$ defined as follows $\Omega:=\{x_{j}\in [0,\pi],\,j=1,2,...,n\}$  and
consider the  Kipriyanov  fractional differential operator     defined in  the paper \cite{firstab_lit:1kipriyanov1960}  by  the formal expression
\begin{equation*}
\mathfrak{D}^{\alpha}f(Q)=\frac{\alpha}{\Gamma(1-\alpha)}\int\limits_{0}^{r} \frac{[f(Q)-f(T)]}{(r - t)^{\alpha+1}} \left(\frac{t}{r} \right) ^{n-1} dt+
(n-1)!f(Q)r ^ {-\alpha} /\Gamma(n-\alpha),\, P\in\partial\Omega,
\end{equation*}
where $Q:=P+\mathbf{e} r,\;P:=P+\mathbf{e}t,\;\mathbf{e}$ is a unit vector having a direction from the fixed point of the boundary $P$ to an arbitrary point $Q$ belonging to $\Omega.$

  Consider the perturbation of the Laplace operator by the Kipriyanov operator
$$
L :=  -D^{2k} + \xi \mathfrak{D}^{\beta} ,\,\mathrm{D}(L)= H_{0}^{k}(\Omega)\cap H ^{2k}(\Omega),
$$
where  $\beta\in (0,1),$
$$
D^{2k}f=\sum\limits_{j=1}^{n} \mathcal{D}_{j}^{2k}f
$$
I was proved in the paper \cite{kukushkin2021a} that
 $$
-C_{0} (D^{2k} f,f)\leq(\mathfrak{Re}Lf,f)\leq -C_{1} (D^{2k} f,f),\;f\in \mathrm{D}(L).
$$
Therefore
$$
\lambda_{n}(\mathfrak{Re} L) \asymp n^{2k/n}.
$$
On the other hand, we have the following eigenfunctions of $D^{2k}$ in the rectangular
$$
e_{\bar{l}} =\prod\limits_{j=1}^{n}\sin l_{j}x_{j},\;\bar{l}:=\{l_{1},l_{2},...,\,l_{ n }\},\,l_{s}\in \mathbb{N},\,s=1,2,...,n.
$$
It is clear that
$$
\;- D^{2k} e_{\bar{l}} =\lambda_{\bar{l}}\,e_{\bar{l}},\;\lambda_{\bar{l}} =\sum\limits_{j=1}^{n}l^{2k}_{j}.
$$
Let us show that the system
$
\left\{e_{\bar{l}} \right\}
$
is complete in the Hilbert space  $L_{2}(\Omega),$ we will show it if we prove that the element that is orthogonal to every element of the system is a zero.
Assume that
$$
 \int\limits_{0}^{\pi}\sin l_{1}x_{1}dx_{1}\int\limits_{0}^{\pi} \sin l_{2}dx_{2}...\int\limits_{0}^{\pi} \sin l_{n}f(x_{1},x_{2},...,x_{n})dx_{n}=  (e_{\bar{l}},f)_{L_{2}(\Omega)}=0.
$$
In accordance with the fact that the system $\{\sin l x\}_{1}^{\infty}$ is a compleat system in $L_{2}(0,\pi),$ we conclude that
$$
 \int\limits_{0}^{\pi} \sin l_{2}dx_{2}...\int\limits_{0}^{\pi} \sin l_{n}f(x_{1},x_{2},...,x_{n})dx_{n} =0.
$$
Having repeated the same reasonings step by step, we obtain the desired result. Taking into account the following   inequality and the embedding theorems, we get
$$
\|\mathfrak{D}^{\beta}f\|\leq C_{\beta}\|f\|_{H_{0}^{1}(\Omega)}\leq C_{\beta,k,n} \|f\|_{H_{0}^{k}(\Omega)},
$$
where
$$
C_{\beta}= \frac{2}{ \beta } \cdot  \|\mathfrak{I}^{1}_{d-} \|  + (1-\beta)^{-1},
$$
We, have
$$
\sum\limits_{l_{1},l_{2},...l_{n}=1}^{\infty} \lambda^{-2}_{\bar{l}}(\mathfrak{Re}L) \|\mathfrak{Im} L e_{\bar{l}}\|^{2}  \leq \xi
C_{\alpha,k,n}\sum\limits_{l_{1},l_{2},...l_{n}=1}^{\infty}\frac{\lambda _{\bar{l}}(  D^{2 })}{\lambda^{ 2}_{\bar{l}}(  D^{2k})}.
$$
Therefore, if the following condition holds
\begin{equation}\label{x14}
\xi
C_{\alpha,k,n}\leq \sum\limits_{l_{1},l_{2},...l_{n}=1}^{\infty}\frac{l^{2}_{1}+l^{2}_{2}+...+l^{2}_{n} }{(l^{2k}_{1}+l^{2k}_{2}+...+l^{2k}_{n})^{2}}<\infty,
\end{equation}
then the conditions of Lemma \ref{L2} are satisfied. Consider the values of the parameters $k,n$ such that the last series is convergent and at the same time $R_{L}\in \mathfrak{S}_{p},\,\inf p=n/2k>1.$ The latter gives us the argument showing   relevance of Lemma \ref{L2} since we can find the range regarding $\alpha$ of the Abell-Lidskii method applicability.

Assume that the following  condition holds
$$
\frac{n}{2}+1<2k<n .
$$
Consider the function
$$
\psi(\bar{l})=\frac{(l^{2k}_{1}+l^{2k}_{2}+...+l^{2k}_{n})^{2}   }{ l^{2}_{1}+l^{2}_{2}+...+l^{2}_{n} },
$$
then $\psi(\bar{t})=nt^{2(2k-1)},\,\bar{t}=\{t,t,...t\}.$ It is clear that the number $s$ of values   $\psi(\bar{l}),\,l_{i}\leq t$ equals to $t^{n},$ i.e. $s=t^{n}.$ Therefore
$$
\psi(\bar{t})=ns^{\frac{2(2k-1)}{n}},\;\psi( \overline{t-1} )=n(s^{1/n}-1)^{2(2k-1) };
$$
$$
n(s^{1/n}-1)^{2(2k-1) } \leq\psi(\bar{l})\leq ns^{\frac{2(2k-1)}{n}},\; t-1\leq l_{i}\leq t,\;i=1,2,...,n.
$$
Having arranged the values in the order corresponding to their absolute value  increasing, we get
$$
 n(s^{1/n}-1)^{2(2k-1) }  \leq\psi_{j} \leq ns^{\frac{2(2k-1)}{n}},\;(s^{1/n}-1)^{n}<j< s,
$$
Therefore
$$
\frac{(s^{1/n}-1)^{2(2k-1) }}{s^{\frac{2(2k-1)}{n}}}<\frac{\psi_{j} }{nj^{\frac{2(2k-1)}{n}}}<\frac{ s^{\frac{2(2k-1)}{n}}}{(s^{1/n}-1)^{2(2k-1) } },
$$
from what follows the convergence of the following series if we take into account the condition  $n/2+1<2k,$ we have
$$
\sum\limits_{j=1}^{\infty} \psi^{-1}_{j}<\infty.
$$
In other words we have proved that the series \eqref{x14} is convergent. Thus, we have considered the case showing relevance of Lemma \ref{L2}. We can claim that the Abell -Lidskii method is not applicable for the values of $\alpha$ less than $n/2k.$

\section{Conclusions}

\end{document}